\documentclass{article}
\usepackage{graphicx}

\DeclareUnicodeCharacter{2212}{\textendash}
\usepackage{geometry}                		
\usepackage{graphicx}				
\usepackage[utf8]{inputenc}								
\usepackage{amssymb}
\usepackage{tcolorbox}
\usepackage{amsmath}	

\usepackage{mathtools}
\usepackage{MnSymbol}
\usepackage[numbers]{natbib} 
\usepackage[colorlinks]{hyperref}
\usepackage{setspace}
\usepackage{physics}
\usepackage{amsthm}

\usepackage{dsfont}

\usepackage{listings}
\usepackage{tikz-cd}

\usepackage{cleveref} 

\usepackage{float}

\usepackage{subcaption}

\newtheorem{theorem}{Theorem}[subsection]
\newtheorem{definition}[theorem]{Definition}
\newtheorem{proposition}[theorem]{Proposition}

\newtheorem{lemma}[theorem]{Lemma}

\numberwithin{equation}{section}

\usepackage{ stmaryrd }

\hypersetup{citecolor={red}} 
 \newgeometry{bottom=3cm}

\newcommand{\q}[1]{``#1''}

\counterwithin{figure}{section}

\usepackage{tikz}
\usepackage{eso-pic}
\usetikzlibrary{calc}
\usetikzlibrary{decorations.pathmorphing}

\title{Are all Weakly Convex and Decomposable 
Polyhedral Surfaces Infinitesimally Rigid?}
\author{Jilly Kevo}
\date{April 2024}

\begin{document}

\maketitle

\renewcommand{\d}[1]{\ensuremath{\operatorname{d}\!{#1}}} 

\begin{abstract}
    It is conjectured that all decomposable 
    (that is, interior can be triangulated without adding new vertices) polyhedra  with vertices in convex position are infinitesimally rigid and
only  recently  has it  been shown that this is indeed true under an additional assumption of codecomposability (that is,
the interior of the difference between the convex hull and the polyhedron itself can be triangulated without adding new vertices). 

 One major set of tools for studying  infinitesimal rigidity happens to be  the (negative) Hessian $M_T$ of the discrete Hilbert-Einstein functional.
Besides its theoretical importance, it provides  the necessary machinery to tackle the problem 
experimentally. 
To search for potential counterexamples to the conjecture, one constructs an explicit family of so-called $T$-polyhedra, all of which are weakly convex and decomposable, while being non-codecomposable.
Since infinitesimal rigidity is equivalent to a non-degenerate $M_T$, one can let  \textit{Mathematica} search for the eigenvalues of $M_T$ and 
 gather experimental evidence that such a flexible, weakly convex and decomposable $T$-polyhedron may not exist.
\end{abstract}

\tableofcontents

\section{Introduction} 

\subsection{Structure of the paper}
Already back in 1766, Leonhard Euler
conjectured that \q{a closed spatial figure
allows no changes, as long as it is not 
ripped apart}. This lead to a lot
of partial results, 
 among them the well-known theorem by Cauchy assuring us that  convex \q{closed spatial figures} are indeed all rigid, in a 
 sense that will be made precise
 soon.
However, the reason 
for why nobody
could prove Euler's conjecture
in its entirety turns out to be that it is simply \textit{wrong}, as Connelly (1977)
\cite{key1}
 demonstrated by providing a counterexample and gifting the world its first flexible and non-self-intersecting polyhedral
surface in $\mathbb{R}^3$. 

Since not all closed surfaces share rigidity, the natural question would be
to seek minimal conditions under which 
the latter holds. This lead to the following 
conjecture, lying at the heart of this
paper: \\ \\
\noindent\fbox{%
    \parbox{\textwidth}{
$\textbf{Main conjecture}$: \\ 
Every weakly convex and decomposable polyhedron is infinitesimally rigid.
}} \\ \\
Let us remark that the motivation for
studying the conjecture stems from 
Izmestiev and Schlenker (2010)
\cite{key8}. There it was 
shown that the statement 
is indeed true under the additional
assumption of 
codecomposability and it was left as
an open problem to
determine whether this
supplementary requirement  
is truly
necessary. The conjecture was however
also mentioned in Connelly and Schlenker (2010) \cite{origin} as Question 1.1. and even before that, in a paper by  
 Schlenker (2005) \cite{origin2} where it also
 originates from. The latter proves that the
 conjecture holds true for all polyhedra $P$ for which there
 exists an ellipsoid containing no vertices of
 $P$ but intersecting all its edges. 
 Before continuing, we'll first need
 some definitions.

\begin{definition}
Let $S$ be a  
triangulation of a compact, orientable
surface with $V$
denoting the
set of vertices and $E$ the set of edges.
A \textbf{polyhedral
surface} or 
\textbf{polyhedron} is the image of a
 map $S \longrightarrow \mathbb{R}^3$
that is affine 
on each edge and  non-degenerate
on the faces.
\end{definition}

\begin{definition}[Izmestiev, 2011, \cite{key6}]\label{system}
    Let $P \subset \mathbb{R}^3$ be a polyhedron 
    with vertices 
    $V = \{p_1, ..., p_n\}$. An \textbf{infinitesimal
  isometric  deformation} 
  of $P$ is a map $q \: : \: 
  V \to \mathbb{R}^3$ such that 
  \begin{equation}\label{eq_1}
     \frac{d}{dt} 
  \biggr|_{t = 0}
  \textnormal{dist}(p_i + tq_i, p_j + tq_j) = 0, 
  \end{equation}
  for all edges $p_ip_j$ of 
  $P$ and where $q(p_i) =: q_i$.
  
\end{definition}

\begin{definition}\label{rigid}
A polyhedron $P \subset \mathbb{R}^3$ is said to be 
\textbf{infinitesimally rigid} if every infinitesimal isometric deformation is trivial 
in first order, that is
$$q_i = K(p_i),$$
for $K$  a Killing field of $\mathbb{R}^3$.
If there is a non-trivial infinitesimal
isometric deformation, the 
polyhedron is said to be 
\textbf{infinitesimally flexible}.
\end{definition}
So, an infinitesimal deformation is just an assignment of vectors to each vertex of a polyhedron $P$. 
If moving the vertices in the assigned directions induces a zero first-order variation 
 of the edge lengths,  we speak of isometric infinitesimal deformation.
Such a deformation is trivial if the Euclidean distance between every pair of points is preserved, that is, the motion is just a  rigid motion in $\mathbb{R}^3$.
Explicitly, one can verify that Equation
\eqref{eq_1} is equivalent to 
$$\langle p_i - p_j, q_i - q_j \rangle
= 0,$$
for all edges $p_ip_j$ of 
  $P$.

\begin{definition}
    A polyhedron $P \subset \mathbb{R}^3$ is said to be \textbf{weakly convex} if every vertex $v$ of
    $P$ has a supporting plane
    that intersects $P$ at exactly $v$.
\end{definition}
For instance, consider a cube having an
additional vertex in the center of one 
of its faces. This polyhedron is convex
but
 not weakly convex. Such a vertex 
 is called \textbf{flat vertex}.
That 
weak convexity is  indeed a necessary condition for infinitesimal rigidity will be illustrated in Section \ref{section3.3} by an example. 

It should also be pointed out that weak convexity does not imply convexity. Indeed, 
the Schönhardt polyhedron, as shown in Figure
\ref{schon}, already provides us with a
counterexample.

\begin{definition}
A \textbf{triangulation} of a 
polyhedron $P$ is a partition of its 
interior into tetrahedra. 
Such a polyhedron is said to be \textbf{decomposable} if its interior can be
triangulated without adding new vertices and 
 \textbf{codecomposable} if the interior of its complement, that is, the difference between the convex hull of $P$ and $P$ itself, can be triangulated without adding new vertices.
\end{definition}
Let us remark that it is 
still an open problem to determine
sufficient criteria a
polyhedron has to satisfy
in order to be 
decomposable. Note that in the discrete
geometry literature, one rather encounters the word 
 \textbf{tetrahedralizable} instead of
 decomposable. 

The structure of the paper is now as follows.
In order to explore the main conjecture,
we would like to 
explicitly test a certain family of weakly convex,
decomposable and non-codecomposable polyhedra 
for infinitesimal rigidity. This is 
done with the aid of the Schönhardt 
polyhedron, which is known to be 
infinitesimally flexible as shown by
Izmestiev (2011) \cite{key6} for instance
and the simplest non-decomposable 
polyhedron as demonstrated by Schönhardt (1928) \cite{key11}. Albeit
not being able to provide a 
counterexample to the main conjecture, 
we'll suggest a recipe that 
enables one to decide 
with a computer whether a given 
polyhedron is infinitesimally rigid or not.
This is done throughout Sections 
\ref{section2} and \ref{section3} with explicit computations given
in the appendix.\\ \\
In Section \ref{section3}, an 
elementary proof of the fact 
that the $\pi/ 6$-twisted
Schönhardt polyhedron 
  is infinitesimally 
flexible is given. 
This is done by rederiving Wunderlich's (1965)
\cite{key13}
formula (Equation \eqref{wunderlich})
by purely geometric means.
Lastly, we provide an 
example to illustrate why weak convexity is a 
necessary condition for infinitesimal
rigidity and conclude with an outlook
summarizing all of the
experimental observations we collected so
far.
This leads to a new conjecture that all
polyhedra belonging 
to a certain family 
and 
satisfying the assumptions 
of the main conjecture must be infinitesimally rigid. Moreover, we collect 
experimental evidence that
for any infinitesimally 
flexible polyhedron 
in that family weak convexity
and 
decomposability can not be 
achieved simultaneously.

\subsection{Regge Calculus and the discretization of space}
Before moving to
the main part of the paper, we'll take
a small detour to motivate the
techniques that are used 
to study the infinitesimal rigidity of
polyhedra. 

In his  paper \q{General relativity without coordinates} (1961) \cite{key10}, Tullio Regge developed a way to 
discretize 
$N$-dimensional Riemannian manifolds 
using 
 a collection of $N$-dimensional building blocks whose intrinsic geometry (their 
 metric)
is Euclidean (that is, flat).
This is known as \textit{Regge Calculus}. Besides the
original paper, see for example 
chapter $42$ of Misner et al. (1973) \cite{key9}. 

Apart from being interesting
for gravitational physics by 
providing the necessary tools to
evaluate the curvature of 
Lorentzian manifolds in an 
intrinsic and manageable way, it  also
constitutes one of the cornerstones
of the mathematical 
formulation  of infinitesimal 
rigidity, considering that the
 $N$-dimensional building blocks mentioned above are Euclidean simplices $S_N$, as 
 depicted in Figure \ref{figS}. \\
\begin{figure}[H]
\centering
\includegraphics[width=5.5cm,height=5.5cm]{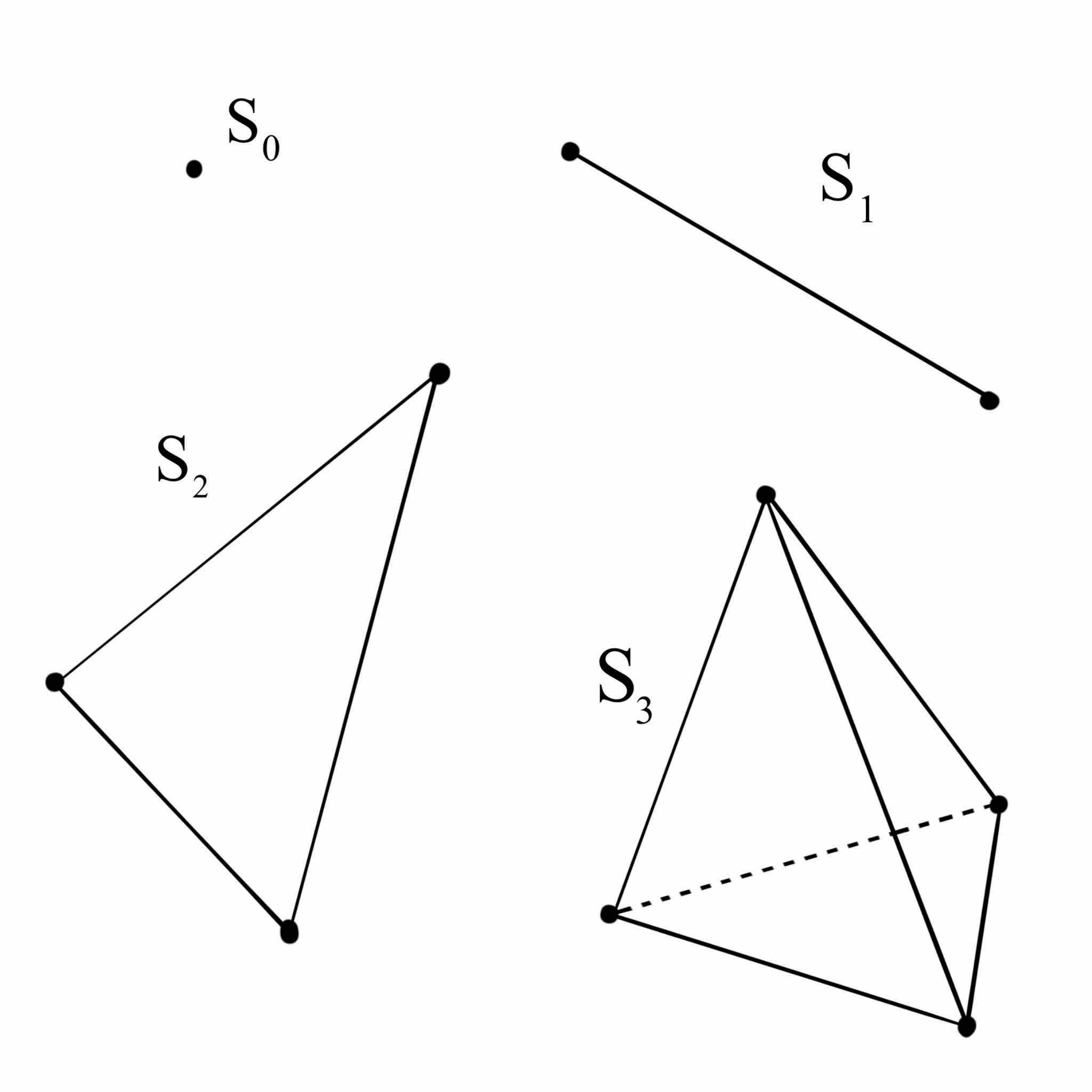}
\caption{A few Euclidean simplices}
\label{figS}
\end{figure}
 Joined facet to facet,  the initially smooth manifold transforms into a 
discrete association of simplices.
What is particularly well-suited 
for computations is that the
characterization of any such discrete manifold skeleton only
requires the specification of the edge lengths of the simplices and the gluing rules for connecting them. 
By choosing a collection of sufficiently small simplices, any  
smooth  manifold can be approximated to arbitrarily  high precision with  an assembly of that kind.
The fact 
 that edge lengths suffice to specify the intrinsic geometry, has been 
 exploited so far to 
 conclude that
\begin{center}
    \textbf{curvature lies concentrated into simplices of dimension $N-2$}.
\end{center}
These simplices are  the so-called \textbf{bones} (or, in the language
of discrete geometry, the $1$-skeleton) 
of the structure and 
turn out to be 
of uttermost value in our context.
Given a weakly convex
polyhedron $P \subset \mathbb{R}^3$  and a triangulation of its interior into $S_3$ simplices such that every vertex of every interior tetrahedron coincides with a vertex of the polyhedron, 
 we will restrain ourselves from moving the vertices of $P$ (and therefore altering its boundary lengths)
and modify the metric inside
of it. Before making all of this precise in the next subsection, 
this is roughly speaking
 done by varying the 
lengths of the 
bones (which are 
just $S_1$ 
simplices) of $P$, that is
 the collection of all interior edges
 of the triangulation. The 
 edge lengths of
 the polyhedron itself are not altered,
 only the ones of the individual tetrahedra.
 As a result, this
   induces a  total angle of
  $\omega_i$ around each such interior 
  edge that potentially differs now from $2\pi$. This is done
  by summing up the individual 
  angles of the now modified 
  tetrahedra that constituted
  the interior triangulation while
  respecting the initial 
  gluing.
 Defining the total curvature of  an interior edge $i$ as $\kappa_i = 2 \pi - \omega_i$, Regge Calculus 
 enables us to express it in terms 
 of edge lengths. 
In order to extend this to infinitesimal rigidity, it is necessary to borrow another function from physics.

\subsection{The discrete Hilbert-Einstein functional
}
Following Izmestiev (2014)
\cite{key7}, we briefly recall the 
most important
 results 
 surrounding
the discrete Hilbert-Einstein functional.

\begin{definition}
    Let $P \subset \mathbb{R}^3$ be a 
    polyhedron and $T$ a triangulation of 
    $P$ with interior edges 
    $e_1, ..., e_n$. Then 
    $\mathcal{D_{P,T}}$ is defined 
    as the collection of $n$-tuples of the
    form $ l := (l_1, ..., l_n) \in \mathbb{R}^n_{>0}$ which, for every
    simplex $\sigma$ of $T$, are such that 
    replacing the lengths of the
    edges of $\sigma$ which are interior
    edges of the triangulation by 
    the corresponding $l_j$'s in $l$
    induces a non-degenerate simplex.
\end{definition}

\begin{definition}
Let $e'_1, ..., e_r'$ denote the boundary 
edges of $P$,  $l_j'$ the length of 
the edge $e'_j$ and  $\alpha_j$ 
the dihedral angle of $P$ at $e'_j$, for
$j \in \{1, ...,r\}$. Moreover, for 
$i \in \{1, ..., n\}$, let 
$\omega_i$ be the total angle 
around the interior edge $e_i$ and 
$\kappa_i := 2 \pi - \omega_i$ the singular
curvature along it. \\
In that case, the \textbf{discrete Hilbert-Einstein
    functional} is defined by
     \begin{align*}
\textnormal{HE} \: : \: \mathcal{D_{P,T}}
&\longrightarrow \mathbb{R} \\
l &\longmapsto        
\sum_{i=1}^{n} l_i \kappa_i + \sum_{j=1}^{r} l_j' (\pi-\alpha_j).
\end{align*}
\end{definition}
Note that this functional can be viewed as the discrete  analog of twice the total scalar curvature of $P$ plus half of the total mean curvature of $\partial P$,
hence its name.
In combination with the $3$-dimensional Euclidean Schläfli formula (which is valid for each individual simplex),
$$ \sum_{e} l_e \d \alpha_e = 0,$$
where the sum runs over all edges of $P$, $l_e$ denoting the length and $ \alpha_e$ the dihedral angle at each edge,
the first-order
variation of the Hilbert-Einstein functional reduces to 
$$   \textnormal{dHE} = \sum_{i=1}^{n} \kappa_i  \d l_i.$$
Note that this 
expression takes in tangent 
vectors to $\mathcal{D_{P,T}}$ 
as input, so that it expresses the 
first-order variation of the interior edge lengths of a triangulation $T$  of a polyhedron $P$ with $n$ interior edges.
Most importantly, the Hessian  $( \pdv{ \textnormal{HE}}{l_i}{l_j } ) $ of 
HE is equal to the Jacobian of the map
$(l_i)_{i=1}^{n} \to (\kappa_i)_{i=1}^{n}$.
Since differentiation
eradicates the constant of $2 \pi$ in the curvature term, one has
$$ M_T := \left( \pdv{\omega_i}{l_j} \right) = - \left( \pdv{ \textnormal{HE}}{l_i}{l_j } \right).$$
Observe moreover that $M_T$ must be symmetric, given that
it equals minus the Hessian of HE. 
This matrix plays an important role in the theory,
as illustrated by the following two results. 
Given that symmetric matrices are
especially well 
suited
for computations, it also provides us with the necessary tools to tackle the problem experimentally. 
\begin{theorem}\label{Th_1}
Let $P$ be a convex polyhedron and $T$ a triangulation admitting $m$ interior and $k$ flat vertices. The dimension of  $\textnormal{ker}(M_T)$ is $3m+k$ and $M_T$ has $m$ negative eigenvalues.
\end{theorem}
 \begin{proof}
Izmestiev and Schlenker (2010) \cite{key8}.
\end{proof}
So, if the convex polyhedron in question can be triangulated  without interior 
and flat vertices, $M_T$ is positive definite. 
Another result that will be of great use to
us is the following:
\begin{lemma}\label{L_1}
Let $P$ be a polyhedron admitting a triangulation $T$ without interior vertices. Then $P$ is infinitesimally rigid if and only if $M_T$ is non-degenerate. 
\end{lemma}
\begin{proof}
Bobenko and Izmestiev (2008) \cite{key0a}. 
\end{proof}

While Izmestiev and Schlenker proved Theorem \ref{Th_1} in \cite{key8}, they obtained the following result as a consequence:
\begin{theorem}\label{Th_2}
If a polyhedron is weakly convex, decomposable, and weakly codecomposable with triangular faces then it must be infinitesimally rigid.
\end{theorem} 
Here, \textbf{weakly codecomposable} denotes any polyhedron $P$ which sits inside a convex polyhedron $Q$ such that  the vertices of $P$ form a subset of the vertices of $Q$ and the complement of $P$ in $Q$ (that is, the difference between $Q$ and $P$)
can be 
triangulated without adding new vertices. 
Recall that codecomposability 
 is stronger in the sense that the complement of the 
polyhedron with respect to its convex 
hull can be triangulated without
adding new vertices.

To see how Theorem \ref{Th_2} follows from Theorem \ref{Th_1}, notice that for a weakly convex, decomposable, and weakly codecomposable polyhedron $P$, there must exist a convex polyhedron $P_c$ such that $P_c$ shares all its vertices with $P$, and 
a triangulation $T$ of $P$ that is contained in a triangulation $T_c$ of $P_c$, where the vertices of $T_c$ are precisely those of $P_c$. 
Now, it can be shown that $M_T$ must be a principal minor of $M_T{_c}$.
Thus, by Theorem \ref{Th_1},  $M_T{_c}$ is positive definite and therefore  $M_T$ must be as well.
Since positive definite matrices are invertible, Lemma \ref{L_1} yields the desired result.

\section{An empirical approach
\label{section2}}
\subsection{The flexible Schönhardt polyhedron}
Depicted in Figure \ref{schon}
is the so-called \textbf{Schönhardt polyhedron}, a polyhedron named after his discoverer Erich Schöhnhardt and having the special property of being weakly convex, non-decomposable and (infinitesimally)
flexible as shown by 
Izmestiev (2011) \cite{key6}.
In fact, it is the 
simplest (in the sense
of fewest vertices) flexible polyhedron which doesn't admit a triangulation without interior vertices as was verified
by Schönhardt (1928)
\cite{key11}, making it 
the perfect footing onto which to construct 
polyhedra violating the codecomposability assumption of Theorem \ref{Th_2}. 
\begin{figure}[H]
\centering
\includegraphics[width=6cm,height=6cm]{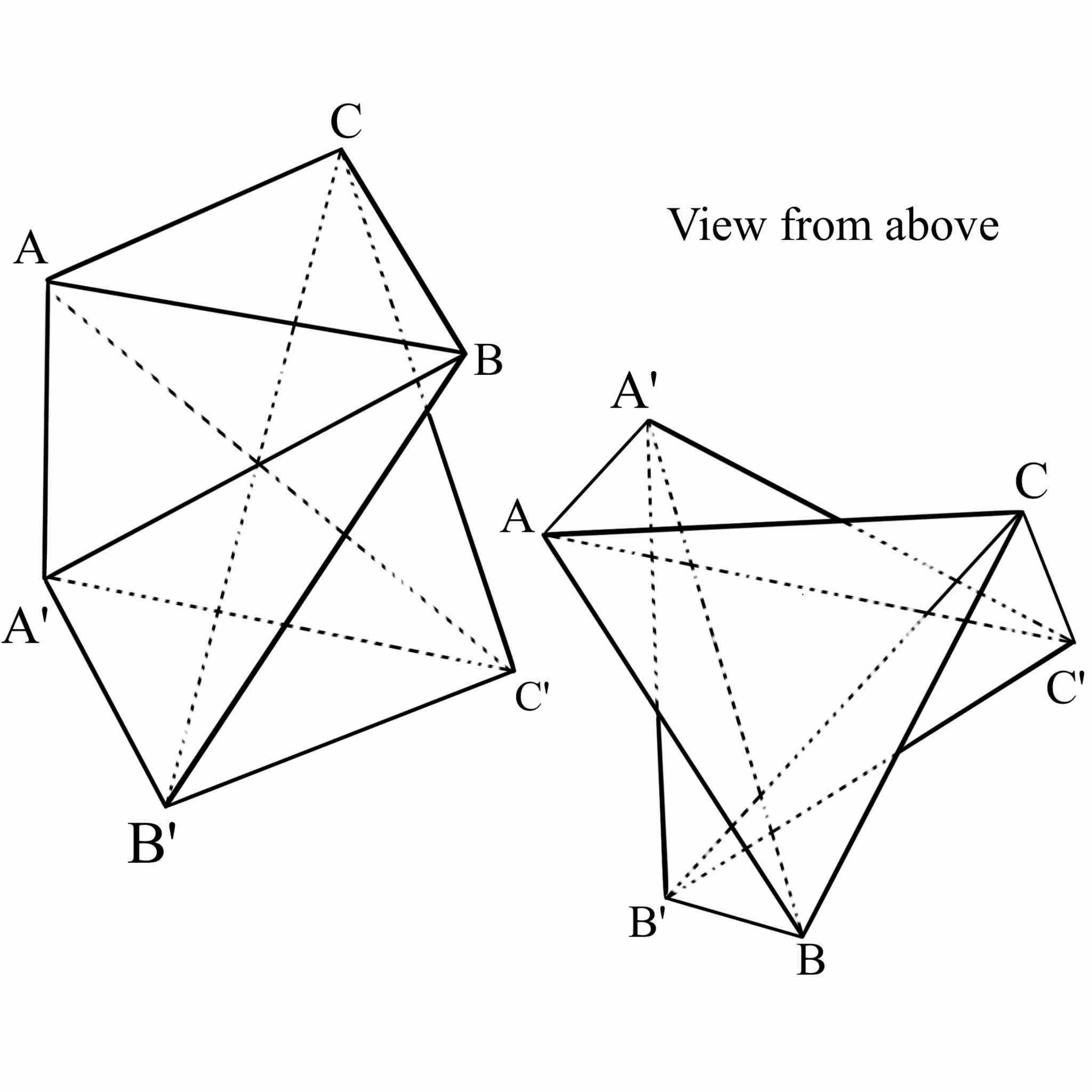}
\caption{The Schönhardt polyhedron}
\label{schon}
\end{figure}
As a figurative example, consider a Schönhardt polyhedron with vertices  \\ \\
\begin{minipage}{0.3\textwidth}
$ A = (1,0,1) \\ \\
 A' =  \left(\cos{\frac{\pi}{6}}, \sin{\frac{\pi}{6}}, -1 \right)
  $
\end{minipage}
\begin{minipage}{0.3\textwidth}
$ B = \left(\cos{\frac{2\pi}{3}},\sin{\frac{2\pi}{3}}, 1\right)
\\ \\
B' = \left(\cos{\frac{5\pi}{6}}, \sin{\frac{5\pi}{6}}, -1 \right) $
\end{minipage}
\begin{minipage}{0.5\textwidth}
$ C =  \left(\cos{\frac{4\pi}{3}}, \sin{\frac{4\pi}{3}}, 1 \right)  \\ \\
C' = (0, -1, -1 ), $
\end{minipage} \\ \\
 \vspace{0.5\baselineskip} \\
in $\mathbb{R}^3$, having edges $AA'$, $AC'$, $BB'$, $BA'$, $CC'$ and $CB'$ and faces $AA'C'$, $AA'B$, $BA'B'$, $BB'C$, $CB'C'$, $CC'A$, $ABC$ and $A'B'C'$.
Later on, we will have to remove faces $ABC$ and $A'B'C'$ from the list since
 we'll restrict  ourselves to 
 weakly convex polyhedra of genus one
 with triangular faces (for the sake of simplicity)
 and admitting the Schönhardt polyhedron as their complement (in the sense that their convex hull contains the Schönhardt polyhedron). 
 
In practice, the following 
three major steps will be 
employed to discern infinitesimally rigid polyhedra from infinitesimally flexible ones:
\begin{enumerate}
  \item Triangulate the polyhedron. This triangulation  $T$ will have $n \in \mathbb{N}$ interior edges.
  \item Calculate all the dihedral angles between the faces of the simplices making up $T$ and meeting at an interior edge of the triangulation of $P$. The total angle around an interior edge will  be the sum of all such dihedral angles.
  \item Determine the eigenvalues of $M_T$ (a $n \times n$ matrix). If zero is not an eigenvalue, $P$ must be infinitesimally  rigid.
\end{enumerate} 
Let us remark that it is in practice not feasible 
 to search for infinitesimal flexible
polyhedra through the linear system of equations 
 \eqref{eq_1} proposed in Definition 
 \ref{system}. In fact, this method would
 require one to stumble exactly upon the right polyhedron and would involve many issues coming
 from rounding errors. However, when working with
$M_T$, this is not the case since
it suffices to just find  two  polyhedra 
 having different signatures. This would then be enough
 to conclude that a polyhedron with
 zero-determinant $M_T$ (which would then be infinitesimally flexible) must exist between them, without
 having to construct it explicitly.

Since we will explicitly construct $P$ (instead of studying random polyhedra), step one is not something we will have to worry about much.
What is of greater importance is to find a way to determine 
the total angle around the interior edges $e_i$, for $ i = 1,..., n$, of $P$ as a function of the edge lengths $l_i$ of the $e_i$ (remember, the entries of $M_T$ are the 
derivatives of total angle around each edge with respect to edge length). 
This is not as straightforward as it sounds at first since 
expressing dihedral angles of tetrahedra  as a function of  edge lengths is 
sometimes
quite cumbersome and unnecessary complicates 
the process even further (for example, using trigonometry and Heron's formula requires calculating the areas of the faces, which we neither need nor want).
A more efficient method to obtain the desired dihedral angles turns out to be through Cayley-Menger determinants.

\subsection{The Cayley-Menger determinant}

A sextuple of the form $S = (e_{12}, e_{13}, e_{14}, e_{23}, e_{24}, e_{34})$ determines a non-degenerate
  Euclidean tetrahedron
  (that is, not all points of $S$ are lying in the same plane) 
  if and only if the following two conditions are satisfied:

\begin{figure}[H]\label{fig2}
\centering
\includegraphics[width=6cm,height=6cm]{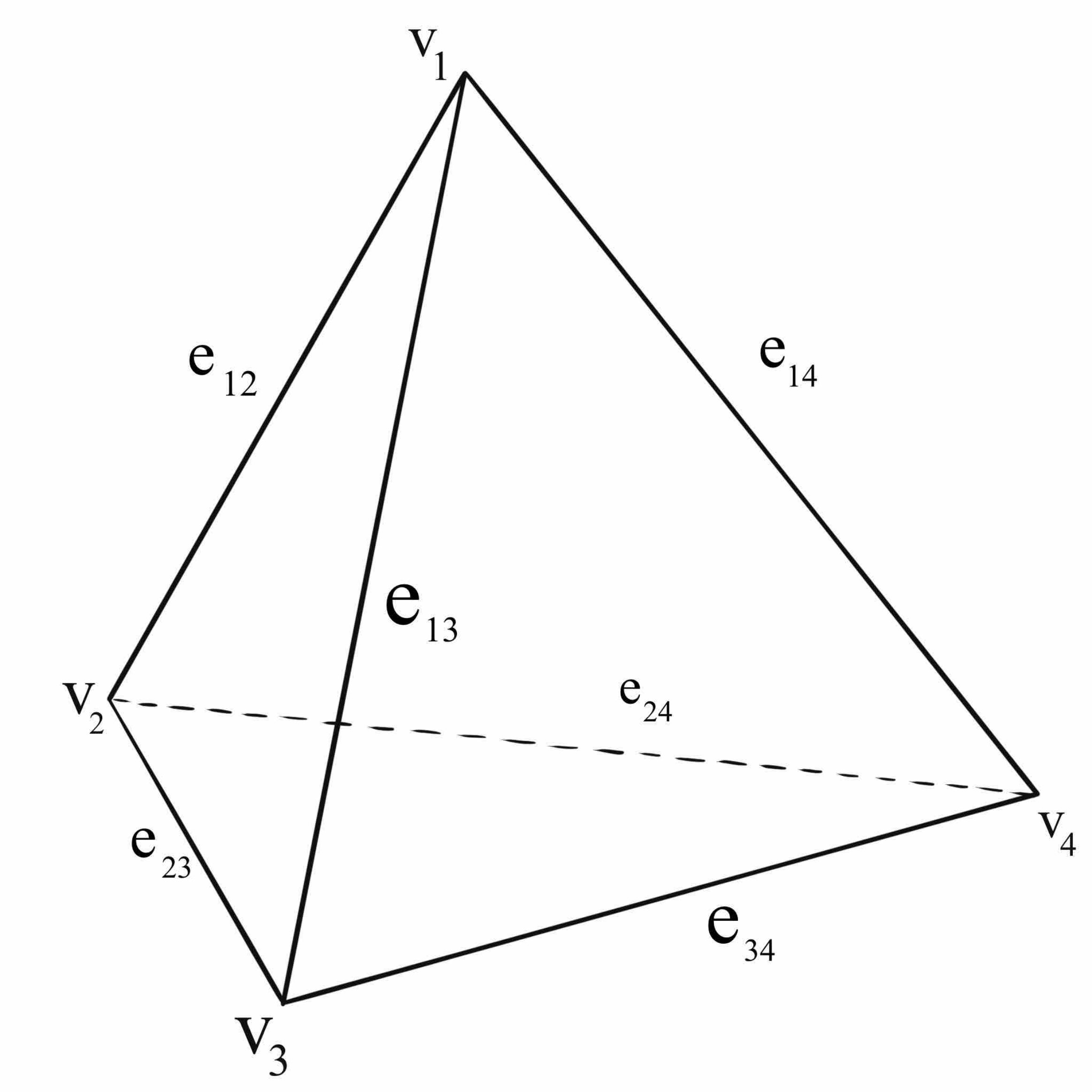}
\caption{A tetrahedron}
\end{figure}
\begin{enumerate}
     \item All face triplets of $S$ are 
     of the form $F = (e_{ij}, e_{ik}, e_{jk})$, where $F$ satisfies $e_{ij} < e_{ik} + e_{jk}$, $e_{ik} < e_{ij} + e_{jk}$ and $e_{jk} < e_{ij} + e_{ik}$;
  \item The determinant $D$ of the following matrix is strictly greater than $0$:
$$ \textnormal{CM} :=  \begin{pmatrix}
0 & e_{12}^2 & e_{13}^2&e_{14}^2 & 1 \\
e_{12}^2 & 0 & e_{23}^2 &  e_{24}^2&1 \\
e_{13}^2 &  e_{23}^2 & 0& e_{34}^2& 1\\
e_{14}^2 & e_{24}^2 & e_{34}^2 & 0 & 1\\
1 & 1 & 1 & 1 & 0
\end{pmatrix}$$
\end{enumerate}

The first condition simply states that the tetrahedron has four faces with non-negative edge lengths satisfying the triangle inequality.
The determinant $D$ is the Cayley-Menger determinant. 
Its geometric significance becomes apparent when considering how it relates to the volume $V$ of the tetrahedron, that is,
$$ D = 288 V^2. $$
See Fiedler (2011)
\cite{key2} for a proof.
In order to compute dihedral angles, it is necessary to look at a specific minor of the matrix CM. 
Given any edge $e_{ij}$ of the tetrahedron, the term $e_{ij}^2$ appears twice in CM, namely in the two rows and columns $i$ and $j$.
To obtain the desired cofactor, defined by
$$D_{ij} := (-1)^{k+l} \cdot (k,l) \: \textnormal{- minor of CM},$$
it suffices to localize the two terms $e_{ij}^2$ in CM and delete row $k$ and 
column $l$ which does not contain the
term $e_{ij}^2$, evaluate the
determinant and
multiply by $(-1)^{k+l}$. It is important that the
$5$th row and column of CM does not take part in 
this excision process.

As an example, in order to obtain $D_{12}$,  delete row 3 and column 4 of CM 
(or, equivalently, row $4$ and column $3$), calculate the determinant of this smaller matrix and multiply it by minus one, since $3 + 4 = 7$. \\ \\
Denoting the interior dihedral angle at the edge $e_{ij}$ by 
$\alpha_{ij}$, the following relationship between $D$ and $\alpha_{ij}$ can be derived:
\begin{equation}\label{eq1}
 \alpha_{ij} = \arccos{ \left( \frac{D_{ij}}{ \sqrt{2e_{ij}^2D+  D_{ij}^2}} \right) }
\end{equation}
There are many ways to prove this, see for instance Fiedler (2011)
\cite{key2}.

\section{Rotational flexibility
\label{section3}}
\subsection{When, how and why does a certain polyhedron twist?}
The first part of this section is devoted
to  the Schönhardt 
polyhedron. In particular, we would like
to understand 
which properties 
discern it from  other, infinitesimally
rigid, polyhedra. 
Let us  therefore start by recalling \\ \\
$\textbf{Cauchy's rigidity theorem}$: If two convex polyhedra in $\mathbb{R}^3$ have pairwise congruent faces, then the two polyhedra must themselves be congruent. \\\\
What is meant by requiring \textbf{congruence} is that the two faces (or sets of points in general) can be transformed into each other by means of isometries, that is, rigid motions in $\mathbb{R}^3$ which are combinations of translations, 
rotations and reflections, with no changes in size allowed.
Even though the untwisted and twisted states of the Schönhardt polyhedron have the same polyhedral net, they form a pair of incongruent octahedra.
Yes, two corresponding faces are congruent to each other and 
two corresponding edges have the same convexity character (one concave edge in each side),
 however, the twisted and untwisted states of the
 Schönhardt polyhedron  do not constitute  convex polyhedra and hence Cauchy's theorem is not applicable to begin with. 
\begin{figure}[H]
\centering
\includegraphics[width=7cm,height=4cm]{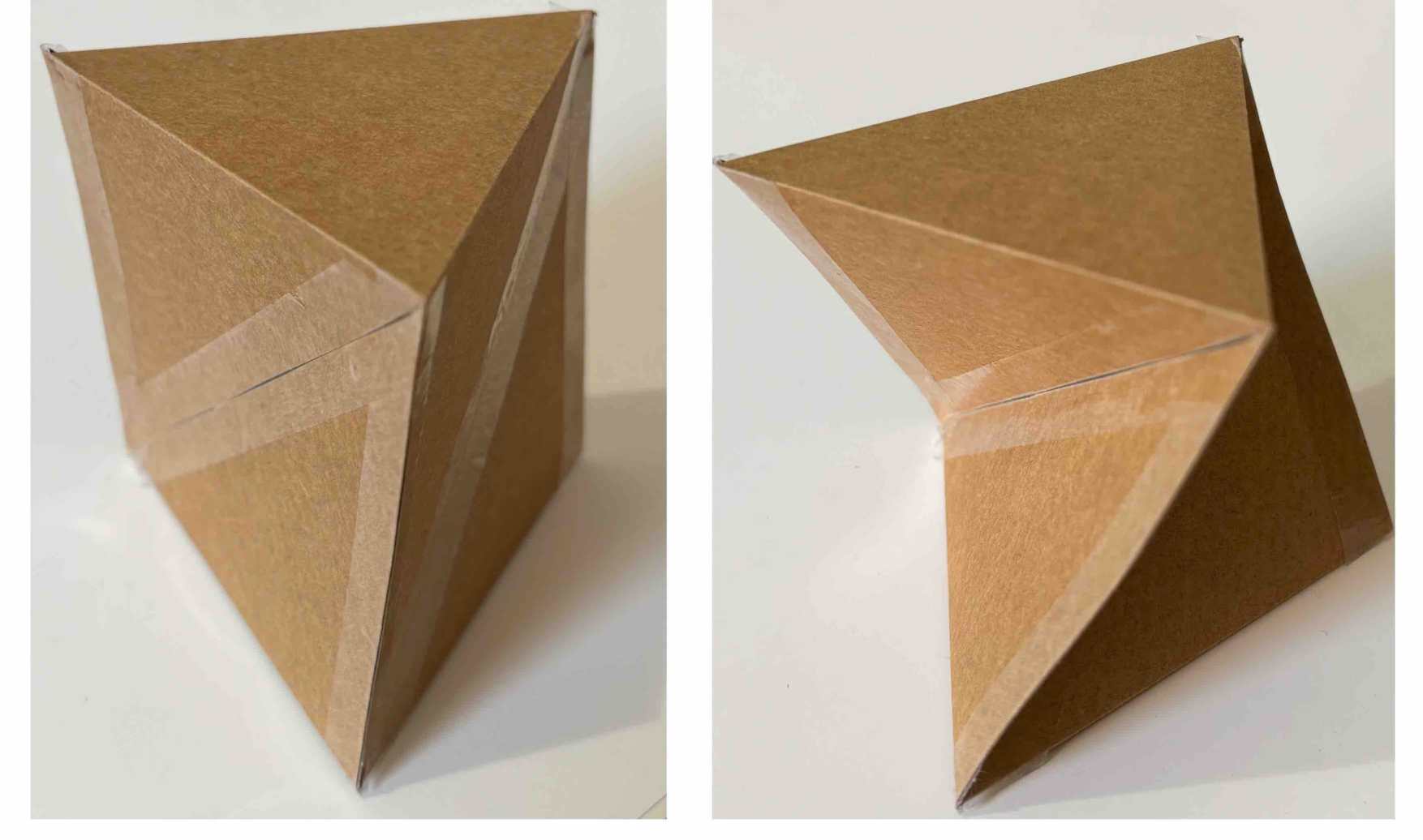}
\caption{Cardboard model of the 
Schönhardt polyhedron}
\label{model}
\end{figure}
As is shown in Figure \ref{model}, the untwisted 
Schönhardt polyhedron is not an upright triangular prism, 
but a so-called concave, triangular gyroprism. Let us also
remark that the two states of the Schönhardt polyhedron can not be continuously deformed into each other. This means in particular that 
if the model depicted
in Figure \ref{model} was
made 
of a perfectly stiff material, one wouldn't be able
to twist it from one state to the other without 
disassembling and rebuilding it. 
Now, if one were to play around with it for a while, one would make the following \textbf{experimental observation} ($\textbf{EO}$): \\ \\
$\textbf{EO1}$: There is no flex of the Schönhardt polyhedron that does not involve a twist. \\ \\
Even though this seems to be a trivial statement, it is far from being a superfluous one. \\
\begin{figure}[H]
\centering
\includegraphics[width=7cm,height=5cm]{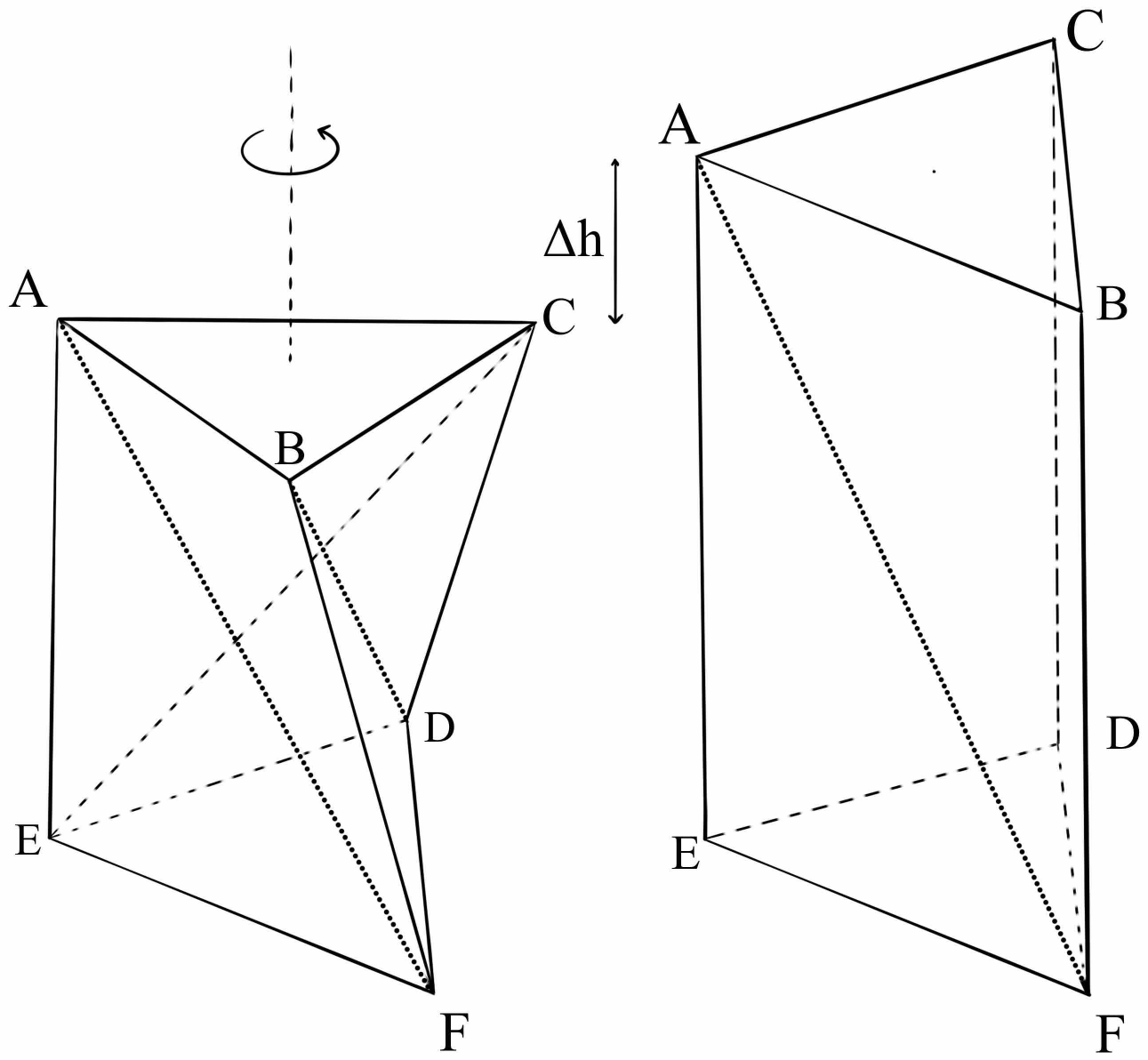}
\caption{Twisting the Schönhardt polyhedron}
\label{twist}
\end{figure}

Still and all,  this is not the end of the story.
Observe how the motion of untwisting the Schönhardt polyhedron induces a height augmentation $\Delta h$ of the structure and that every vertex of the top triangle is prescribed to move on a cylinder 
(while the bottom basis is kept in a fixed position). 
It is then natural to believe
that this $\Delta h$ can be 
expressed in terms of some of the
 vital polyhedral characteristics 
(such as height, angle and edge lengths of the two equilateral triangles). \\
Moreover, one can easily convince oneself,
by trying to construct a model for
instance, 
that
 the polyhedral net must not be made out of a perfect rectangle but rather a parallelogram.
We can denote this little overhang by $m$,
as depicted in Figure \ref{net}.

\begin{figure}[H]
\centering
\includegraphics[width=10cm,height=4cm]{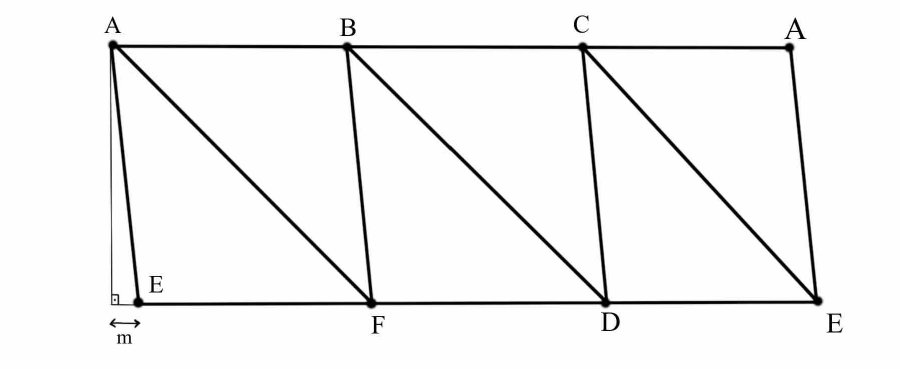}
\caption{The Schönhardt polyhedron's net with
top and bottom faces removed}
\label{net}
\end{figure}

As already remarked, we can keep the bottom basis fixed and imagine an infinitesimal flex.
During that process, all edge lengths are kept at constant lengths and so the vertex $A$ (for instance) of the top basis is not only constrained to move on a cylinder of radius $r$, but also on a circle of radius $c$, 
(corresponding to the edge length $EA$). Figure 
\ref{fig4} describes how the (infinitesimal)
twist is done:

\begin{figure}[H]
\centering
\includegraphics[width=13cm,height=12cm]{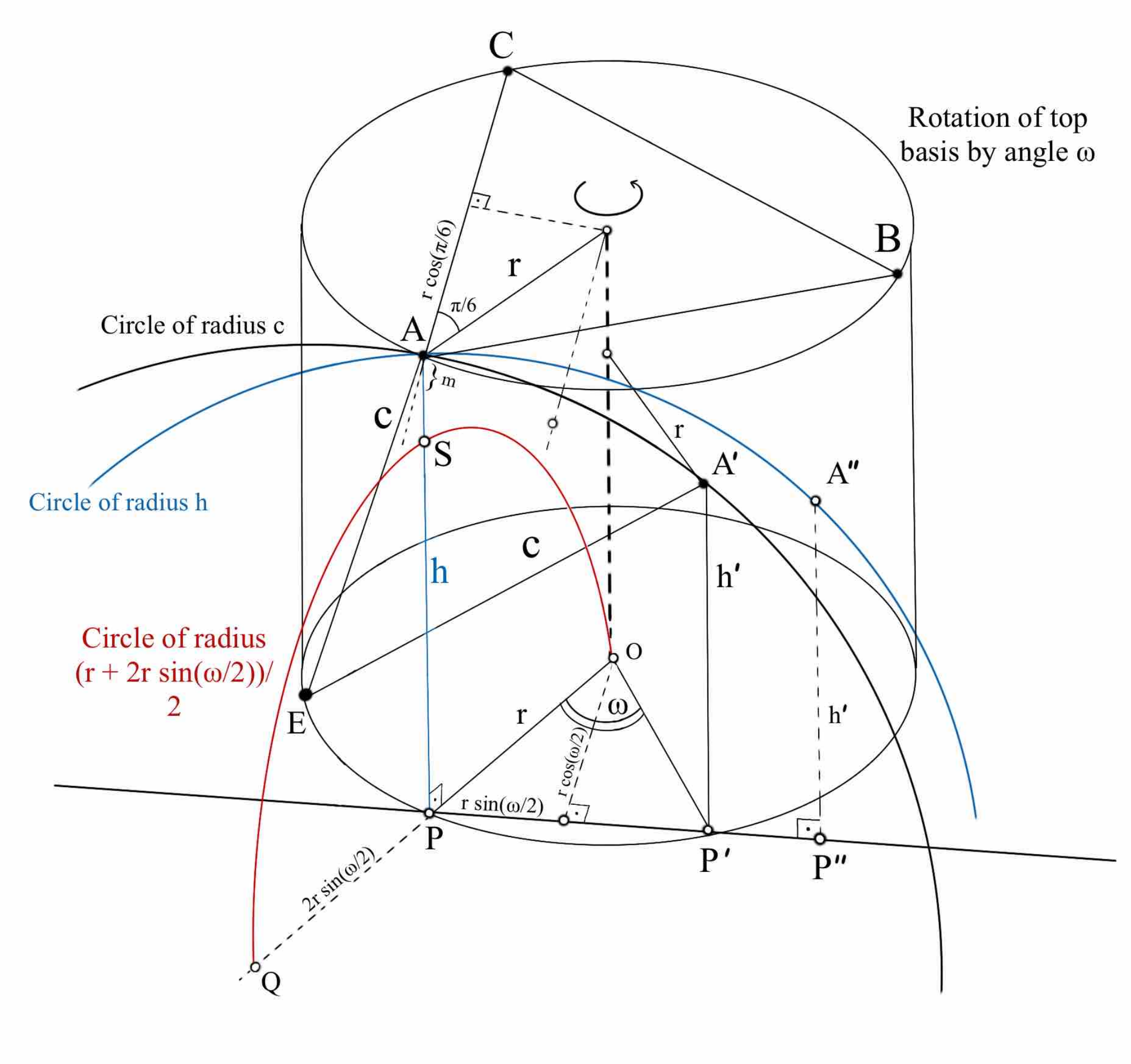}
\caption{Twisting the Schönhardt polyhedron}
\label{fig4}
\end{figure}

\begin{proposition}
    With notation as in Figures  
  \ref{net} and \ref{fig4}, the difference of the square  
    of the heights of the  twisted and
untwisted Schönhardt polyhedron denoted, respectively, by $h$ and 
$h'$, is given by
\begin{equation}\label{wunderlich}
    h^2 - h'^2 = 2r^2 \sin{ \left(\frac{\omega}{2} \right)}.
\end{equation}
\end{proposition}
\begin{proof}
Note that this formula already appears 
in Wunderlich (1965)
\cite{key13}, albeit 
derived by different means. 
Here we present a purely 
     geometric proof
    using Figure \ref{fig4}.
   Hollow points  indicate that they do not belong to the set of vertices of the Schönhardt polyhedron. \\
   Now, 
since $ABC$ is an equilateral triangle, all its angles must be equal to $\pi/3$.
Thus, by construction, the two segments of length $r\cos{\left( \frac{\pi}{6} \right)} $ and $r\cos{\left( \frac{\omega}{2} \right)} $ are parallel and their difference is equal to $m$, that is,
$$ m = r\cos{\left( \frac{\omega}{2} \right)}  - r\cos{\left( \frac{\pi}{6} \right)}  = r\left(\cos{\left( \frac{\omega}{2} \right)} - \cos{\left( \frac{\pi}{6} \right)}\right). $$
With the fact that cosine is bounded by $1$ and $1 -  \cos{\left( \frac{\pi}{6} \right)} = 1 - \sqrt{3}/2 \approx 0.134$, we find that the overhang $m$ must satisfy
$$ 0 \leq m \leq 0.134 \cdot r.$$
On the other hand, observe that
$$ \textnormal{distance}(P, P') = 2r \sin{ \left(\frac{\omega}{2} \right)}.$$

Constructing the blue circle of radius $h$ and center $P$ gives birth, upon projecting the point $A'$ onto the blue circle and connecting it to the line that passes through the
points $P$ and $P'$, to the right angled triangle $PA''P''$.
In order to find the distance between $P$ and $P''$, we can extend the line segment connecting $O$ and $P$ by a length of $ 2r \sin{ \left(\frac{\omega}{2} \right)}$.
The trick consists in taking $2r \sin{ \left(\frac{\omega}{2} \right)} + r$ to be the diameter of a new circle. 
This circle (depicted in red) has of course radius $ \left(2r \sin{ \left(\frac{\omega}{2} \right)} + r \right)/2$ (which is the arithmetic mean of $2r \sin{ \left(\frac{\omega}{2} \right)}$ and $ r$) and intersects the segment $h$ at the point 
$S$ which, by the geometric properties of semi-circles, is at a distance of $\sqrt{2r^2 \sin{ \left(\frac{\omega}{2} \right)}}$ from the point $P$.
The length of $\sqrt{2r^2 \sin{ \left(\frac{\omega}{2} \right)}}$ corresponds to the geometric mean of $2r \sin{ \left(\frac{\omega}{2} \right)}$ and $ r$.
The only thing that is left to do is to notice that 
$$ \textnormal{distance}(P, P'') = \textnormal{distance(P, S)} $$
and apply the Pythagorean theorem to the triangle $PA''P''$
$$h^2 = h'^2 + \left(\sqrt{2r^2 \sin{ \left(\frac{\omega}{2} \right)}}\right)^2,$$
that is,
$$h^2 - h'^2 = 2r^2 \sin{ \left(\frac{\omega}{2} \right)}.$$
\end{proof}
It remains to show that the Schönhardt polyhedron is indeed infinitesimally flexible.
\begin{theorem}
    The $\pi/6$-twisted Schönhardt
    polyhedron is infinitesimally
    flexible.
\end{theorem}
\begin{proof}
By Definition \ref{rigid}, we know that a polyhedron is infinitesimally rigid if every isometric  infinitesimal deformation is trivial in first order, so that
 each motion corresponds to a rigid 
body motion.
If $\nu = \{v_1,..., v_n\}$ 
denotes the set of vertices 
of the polyhedron $P$, any such motion 
can be expressed through a map
$q : \nu \to \mathbb{R}^3$ satisfying
\begin{equation}\label{eq3}
 \frac{d}{dt}\biggr|_{\substack{t = 0}} \textnormal{distance}(v_i + tq(v_i),   v_j + tq(v_j)) = 0 
 \end{equation}
for every edge $v_iv_j$ of $P$.
This is equivalent to 
\begin{equation}
\langle v_i - v_j, q_i - q_j \rangle = 0
\end{equation}
because \eqref{eq3} forces  $2(v_i + tq(v_i)q_i + 2(v_j + tq(v_j)q_j = 0$ which, at $t = 0$, is just $v_iq_i + v_jq_j = 0$. 
    So, in order to show infinitesimal
    flexibility, we need to find at least one  infinitesimal isometric deformation that is non-trivial in first order. 
    
In the case of the Schönhardt polyhedron, $\nu = \{A, B, C, D, E, F\}$. Assume furthermore that the top basis is twisted by an angle of $\theta = \pi/6$ with respect to the bottom basis.
Here, $\theta$ denotes the total rotation, that is, the angle $\omega$ (as 
depicted in Figure \ref{fig4}) and the smaller angle stemming from the overhang $m$ added together.
The reason for choosing the particular value of $ \pi/6$ is that this corresponds
 to the unique choice of \q{twisting angle} $\theta$ that makes the polyhedron infinitesimally flexible. 
 
To see why, assume for simplicity that the side lengths of the equilateral triangles $ABC$ and $DEF$ are equal to $1$ and apply the Pythagorean theorem to the  triangle $AFE$ such to obtain
$$AF^2 = 1^2 + \left(\sqrt{ \frac{ \sin{( \pi/3 + \theta)}}{\cos{(\pi/6)}}}\right)^2 $$
and 
$$AF^2 = 1 + \frac{2}{\sqrt{3}} \sin{( \pi/3 + \theta)}. $$
This function hits its maximum at $\theta = \pi/6$ and therefore forces its derivative to have a zero at that value.
In other words, the edge length $AF$ does not change (up to first order) with respect to $\theta$, or, by imagining that the twist is executed in a uniform and symmetric manner, $AF$ is  kept constant with respect to time
(similar holds of course for the other diagonals $BD$ and $CE$).
Of course,
this will then be used 
to study infinitesimal isometric 
deformations. 

Let $q(D) = q(E) = q(F) = 0$ or, equivalently,  keep the bottom basis of the polyhedron 
 at a fixed position and apply a non-zero velocity vector (pointing to the outside of the polyhedron) to the remaining vertices such that $q(A)$ is
 orthogonal to the plane $AEF$ and $q(B)$ and $q(C)$ are the images of $q(A)$ under rotation by an angle of $2\pi /3$ and $4\pi/3$ around the axis of the cylinder of radius $r$.
Clearly, the edge lengths of the triangle $ABC$ are kept constant since no transformation is applied to its vertices.
Since $q(A)$ is just an infinitesimal rotation around the edge $EF$ of the vertex $A$, the side lengths of the triangle $EAF$ are preserved up to first order.

By symmetry, $q(B)$ and $q(C)$ have a similar effect on the vertices $B$ and $C$ and thus the side lengths of the triangles $DFB$ and $CDE$ are also infinitesimally preserved.
Taking  our particular choice of $\theta = \pi/6$ into account, the planes $AEF$, $BDF$ and $CDE$ must pass through the center of the triangle $ABC$ and so $q(A)$, $q(B)$ and $q(C)$ are tangent to the cylinder of radius $r$ and 
 axis passing through the center of $ABC$.
This on its own means, by taking the symmetry of the motion into consideration, that $ABC$ does indeed \textit{twist} around the axis of the cylinder and therefore naturally preserves its side lengths. 

Hence, we have found a non-trivial infinitesimal isometric transformation (corresponding to a twist) and can conclude that this Schönhardt polyhedron (of angle $\theta = \pi/6$) is indeed infinitesimally flexible. 
\end{proof}
Remember, our initial  goal is to build a weakly 
convex and decomposable structure violating the codecomposability hypothesis of the 
main conjecture and, in the (quite utopian) best-case  scenario, obtain an infinitesimally  flexible polyhedron that will therefore disprove it. \\ \\
Now, the absence of a suitable triangulation for the Schönhardt polyhedron  is due to the fact that any such tetrahedron would have vertices contained in the top and bottom basis of the
latter and therefore edges that coincide with the diagonals of the Schönhardt polyhedron.
However, the Schönhardt polyhedron has  no internal diagonals and so any simplex of a potential triangulation could never lie entirely inside, thus the impossibility to decompose it. \\ \\
Before investigating more complicated arrangements,
 it will prove useful to set up some terminology.

\subsection{Some weakly convex, non-codecomposable polyhedra}
Let us start by remarking that whenever we 
mention notions of rigidity and flexibility in
what follows, we
are referring to infinitesimal rigidity and
infinitesimal flexibility. 

Since we would like to examine 
the family of weakly convex, decomposable polyhedra having the Schönhardt polyhedron as its complement, we 
set up some terminology and  refer to them as $\textbf{$T$-polyhedra}$, $T$ 
standing for \textit{twist} (even though most of them will be perfectly rigid), 
Every such polyhedron can be 
partitioned into  three regions: 

\begin{figure}[H]
\centering
\includegraphics[width=6.5cm,height=6.5cm]{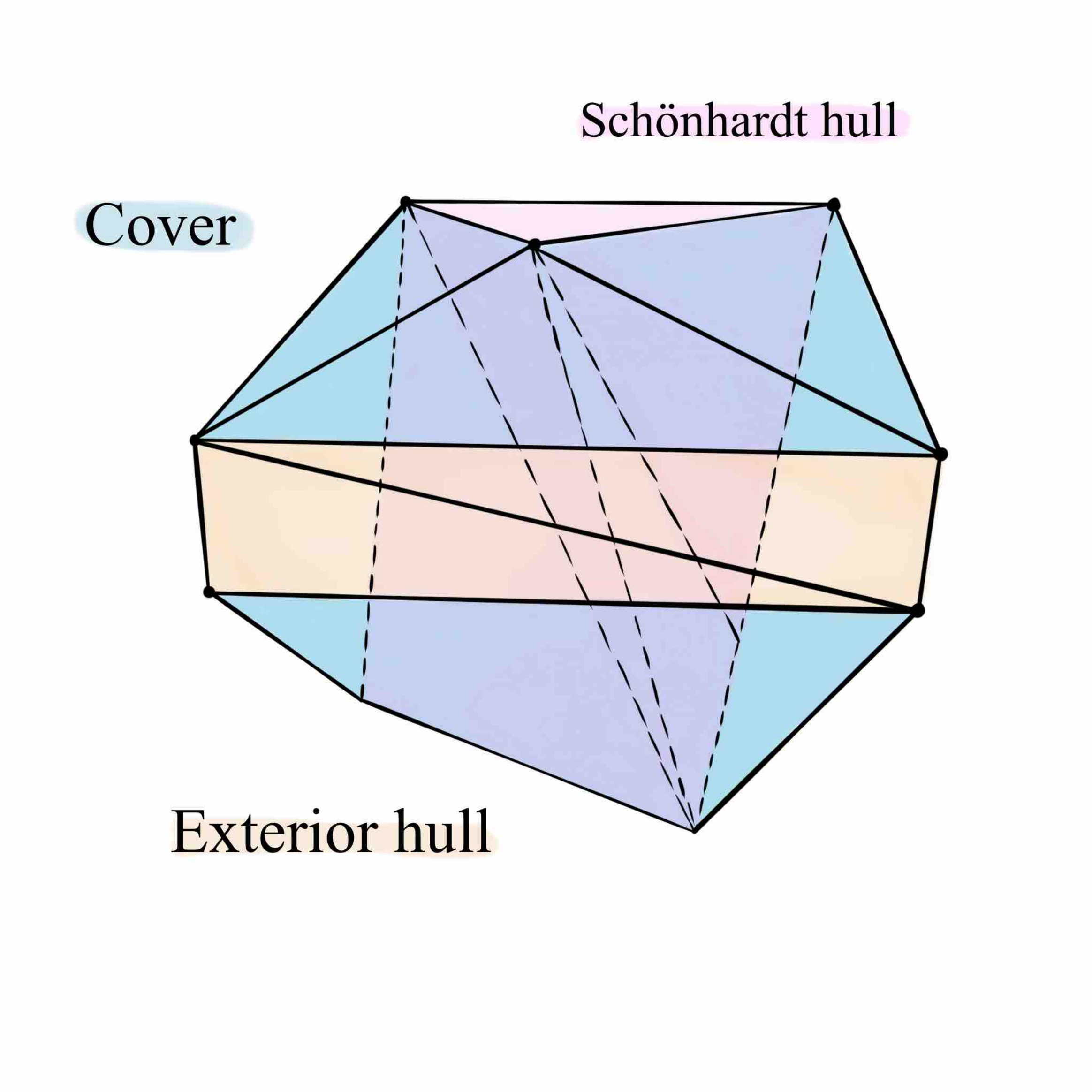}
\caption{A $T$-polyhedron}
\label{T-poly}
\end{figure}
The \textbf{cover} at the top and bottom, the \textbf{Schönhardt hull} lying in the middle of the structure and the \textbf{exterior hull}, a polygonal ring connected to the cover.
Notice how vertices belonging to the cover must necessarily be connected to vertices of the Schönhardt hull.
Of course, nothing prevents the exterior hull from being trivial (in the sense of nonexistent) such that all vertices of the upper cover are immediately connected to the bottom cover.
However, polyhedra of that kind will not be of great use to us since they most certainly will never twist. 
To see why, observe the following: \\\\
$\textbf{EO2}$: Any potential infinitesimal
twist of a $T$-polyhedron will preserve the geometry of the cover, that is, the cover can only perform an infinitesimal
rigid body motion. \\ \\
In other words, the distance between vertices of the cover and vertices of the Schönhardt hull to which they are connected remains constant during any potential twist. 
Thus, if the exterior hull is trivial, the rigidity of the cover will prevent the polyhedron from twisting. 
One could speak of \textit{induced} rigidity stemming from the exterior hull and spreading onto the whole $T$-polyhedron. In that spirit: \\ \\
$\textbf{EO3}$: If the exterior hull is rigid (can not be twisted), then the $T$-polyhedron is itself rigid (can not be twisted either).  \\ \\
So, in the case of convex exterior hulls, it is not necessary  to calculate any eigenvalues or Cayley-Menger determinants since $\textbf{EO3}$ allows us to immediately conclude   that the polyhedron is rigid, which already rules out 
quite a few 
candidates. Alright, but there is even more. \\ \\
With $\textbf{EO1}$ in mind, note how the rotational flexibility of the Schönhardt polyhedron forces the exterior hull to not only  have rotational flexibility as well, but to rotate in one and the same manner. 
For instance, constructing a $T$-polyhedron with an exterior hull that rotates to the right and a Schönhardt hull that twists to the left  would 
produce a
 perfectly rigid structure 
 even though the 
individual \q{building pieces} are not.
The same is (probably?) true for exterior hulls that would potentially flex in some other way that does not involve a twist. \\ \\
In the spirit of the discussion from the previous section, 
recall how the motion of untwisting the Schönhardt polyhedron induces a height augmentation $\Delta h$ of the structure.
This implies that any exterior hull giving  rise to a flexible $T$-polyhedron would have to vary in height by the same amount $\Delta h$. Thus: \\\\
$\textbf{EO4}$: A flexible $T$-polyhedron is characterized by the height variation $\Delta h $ of the Schönhardt hull. Any rotational motion of the exterior hull must not only be performed in the same direction as the one of the 
Schönhardt hull, but vary in height by the same amount of $\Delta h$. \\ \\
Being equipped with $\textbf{EO4}$ and the contrapositive of $\textbf{EO3}$, a natural choice of $T$-polyhedra to test would be the ones obtained by replacing the exterior hull by a second Schönhardt polyhedron.
After all, we can be assured that the exterior hull has the same \q{rotational properties} as the Schönhardt hull and calibrating the edge lengths of the equilateral triangles constituting it with respect to the height of the twisted exterior hull 
(such that to 
verify $\textbf{EO4}$), this would supply us with a nice and flexible polyhedron. See Figure \ref{fig:fig}, with
 top and bottom cover removed 
 for a better visualization. 
\begin{figure}[H]
\begin{subfigure}{.5\textwidth}
  \centering
\includegraphics[width=4.5cm,height=3.5cm]{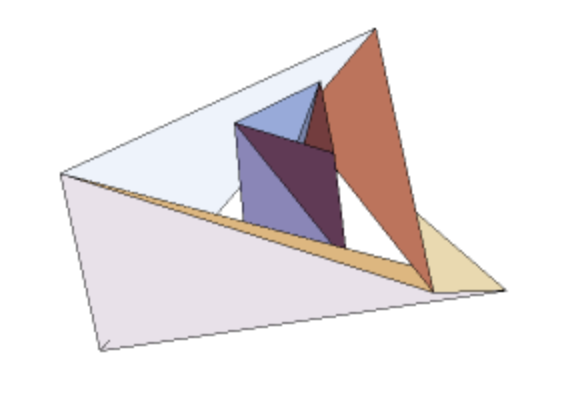}
  \label{TS3}
\end{subfigure}%
\begin{subfigure}{.5\textwidth}
  \centering
  \includegraphics[width=3.5cm,height=3cm]{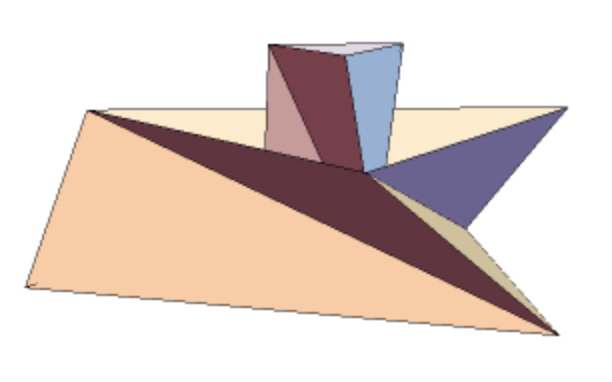}
  \label{TS2}
\end{subfigure}
\caption{Nested Schönhardt polyhedra}
\label{fig:fig}
\end{figure}

Even though this might seem like we  completed  our task by finding a flexible $T$-polyhedron (and the thought does indeed admit quite a compelling attraction), we are unfortunately anything but done.
The problem lies in the sole fact that our promising candidate is actually not   a $T$-polyhedron! \\ \\
To see why, remember that $T$-polyhedra are by construction assumed to be decomposable and that calculating eigenvalues of $M_T$ requires us to find a triangulation of the polyhedron as a preliminary.
Since there must exist some simplex of any potential triangulation of the preceding polyhedron that would have to share one of its faces with one of the triangular faces of the Schönhardt hull, the remaining
fourth vertex of the simplex would be forced to lie on a vertex of the exterior hull, creating  thereby  forbidden intersections. 
\begin{figure}[H]
\begin{subfigure}{.5\textwidth}
  \centering
  \includegraphics[width=3.5cm,height=3cm]{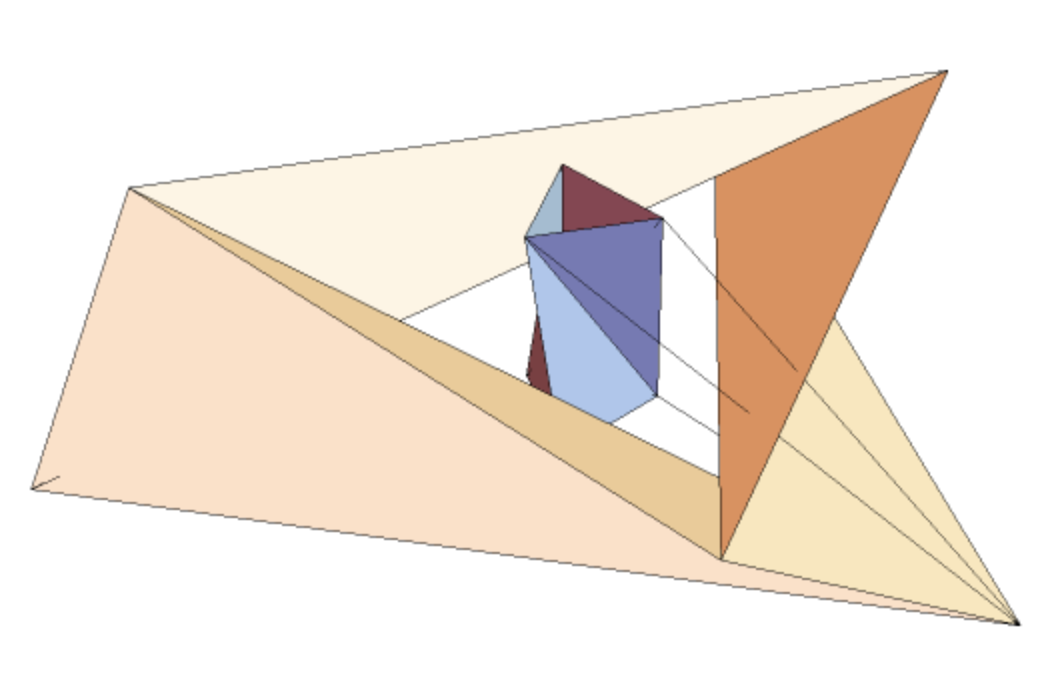}
  \label{fig:sfig1}
\end{subfigure}%
\begin{subfigure}{.5\textwidth}
  \centering
  \includegraphics[width=4cm,height=3.5cm]{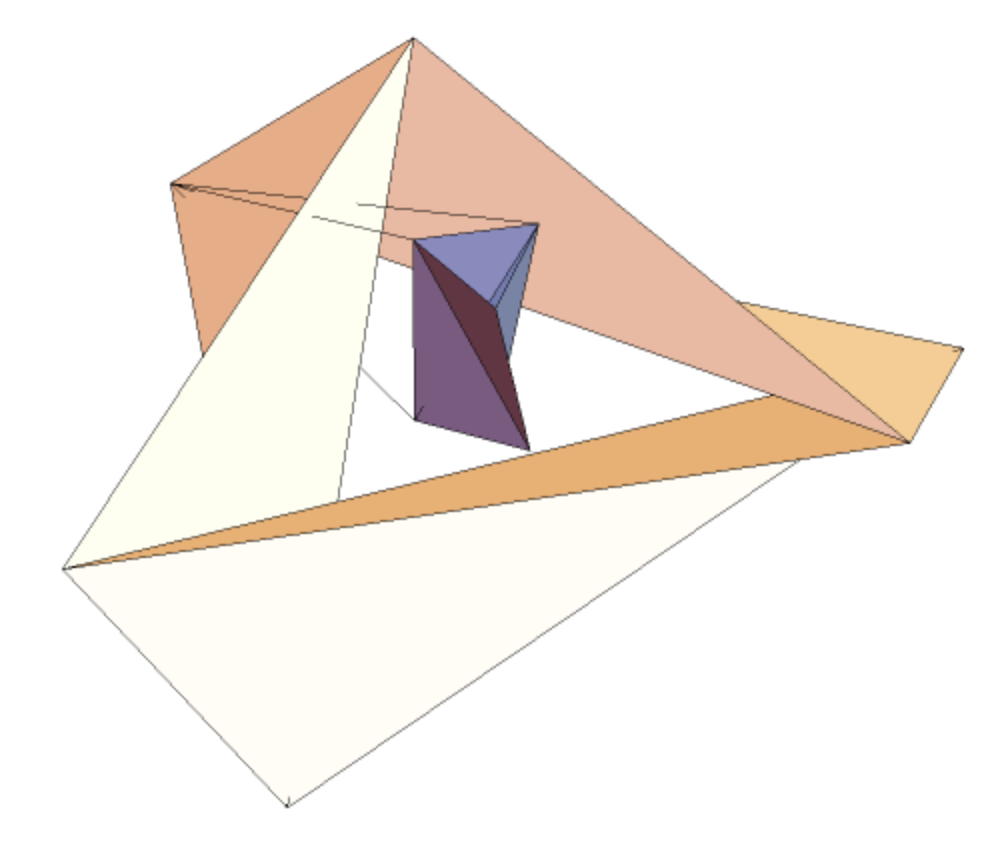}
  \label{fig:sfig2}
\end{subfigure}
\caption{Resulting non-decomposable polyhedron}
\label{fig:fig4}
\end{figure}
Since the possibly  easiest choice of exterior hull that permits rotations did not provide us with the result we sought, it is time to look at more complicated (and less obvious) examples with the help of a computer.
This purely computational approach will however not be explored in depth here (see Appendix).

\subsection{Seemingly  incompatible assumptions \label{section3.3}}
It proves not to be particularly difficult to find counterexamples to variations of the main
conjecture obtained by dropping one of the assumptions. To see why weak convexity 
is necessary, consider
the polyhedron in Figure \ref{poly}.
Since it is decomposable via a 
certain triangulation $T$, one can
 calculate the eigenvalues of
the 
associated matrix $M_T$.
Of course, such a computation is not needed in this particular case since Cauchy's Rigidity Theorem assures us that due to its convexity, the polyhedron must be rigid.
As a check, it 
can be 
computationally confirmed  
that $M_T$ does indeed admit only positive eigenvalues.

\begin{figure}[H]
\begin{subfigure}{.5\textwidth}
  \centering
  \includegraphics[width=5cm,height=5cm]{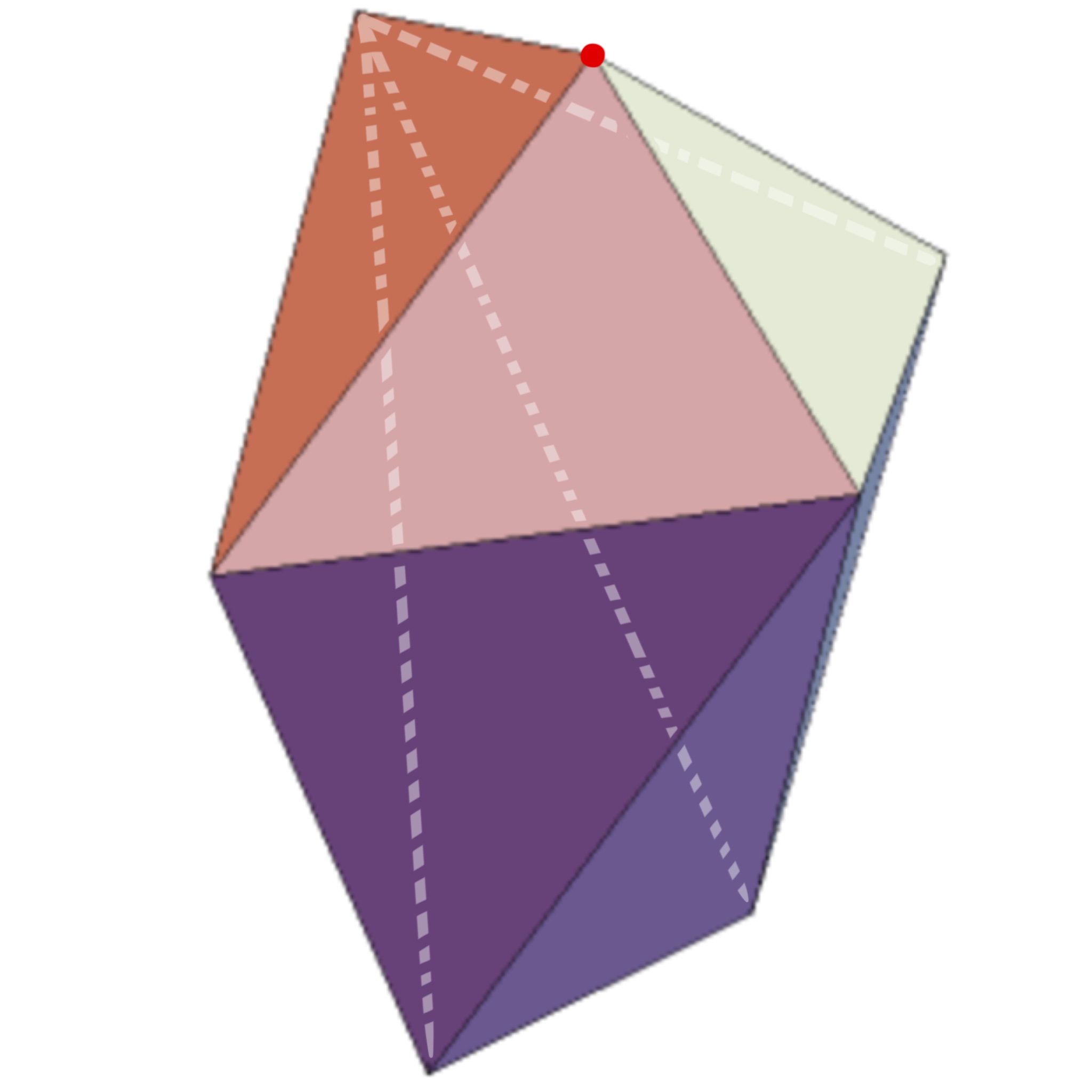}
  \caption{A convex polyhedron}
  \label{poly}
\end{subfigure}%
\begin{subfigure}{.5\textwidth}
  \centering
  \includegraphics[width=4cm,height=5cm]{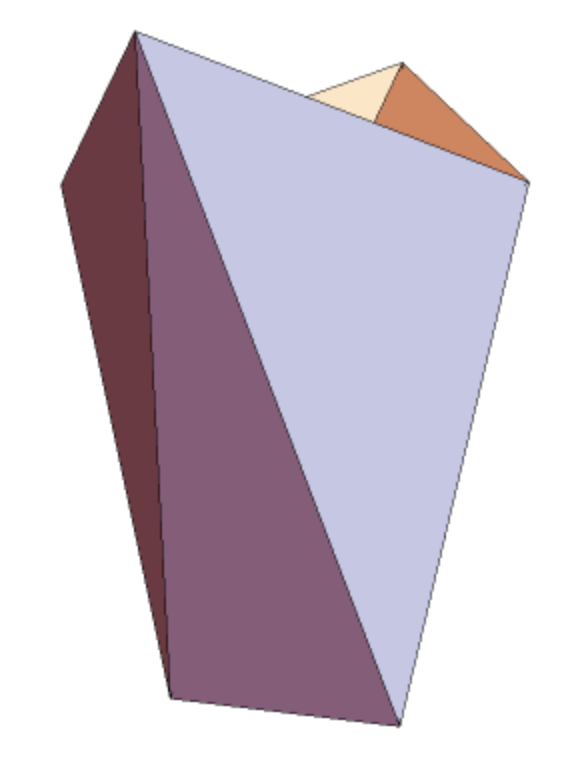}
  \caption{Pushing the red vertex inside}
  \label{poly2}
\end{subfigure}
\begin{subfigure}{.5\textwidth}
  \centering
  \includegraphics[width=4.5cm,height=4.5cm]{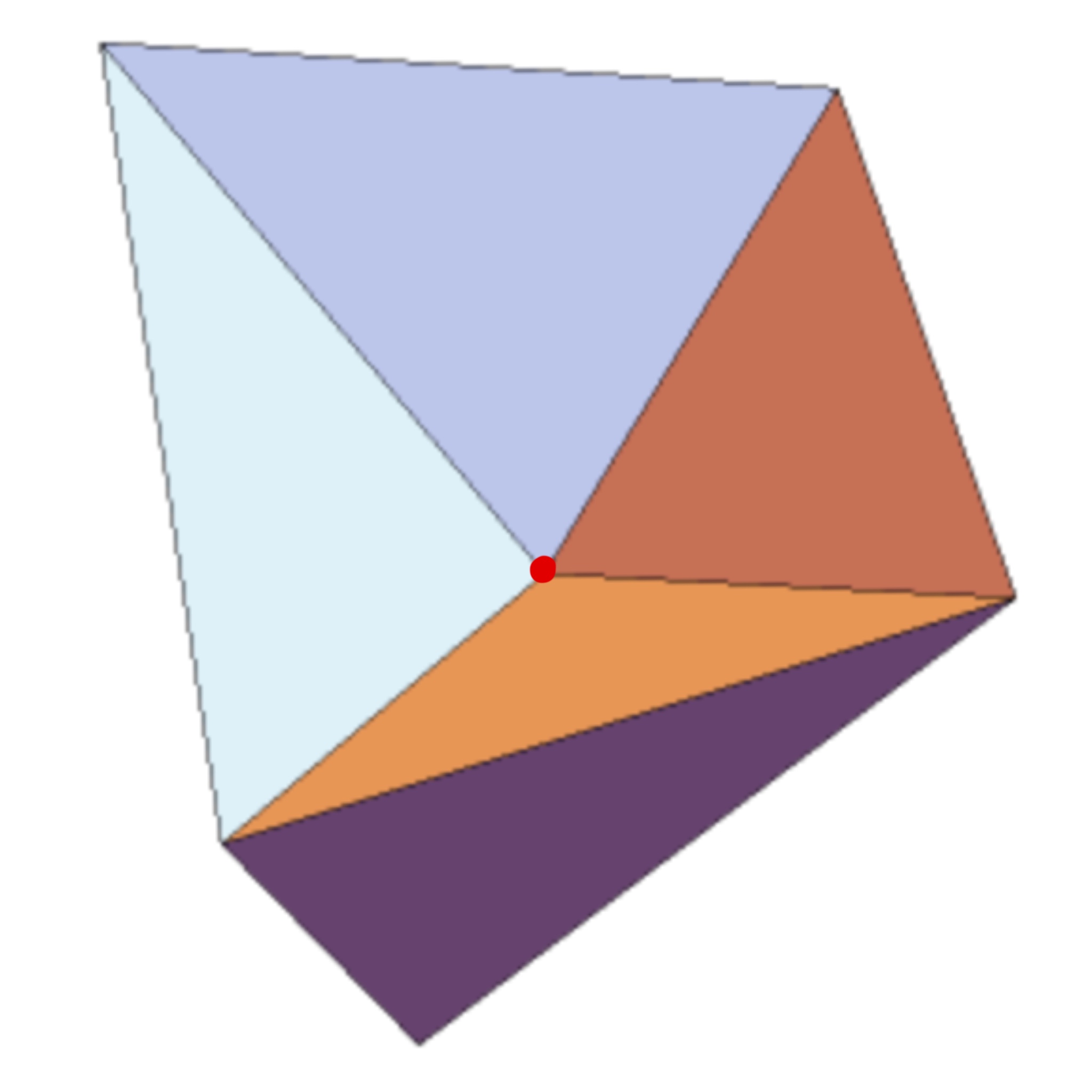}
  \caption{View from above}
  \label{fig:sfig3}
\end{subfigure}
\begin{subfigure}{.5\textwidth}
  \centering
  \includegraphics[width=4.2cm,height=5cm]{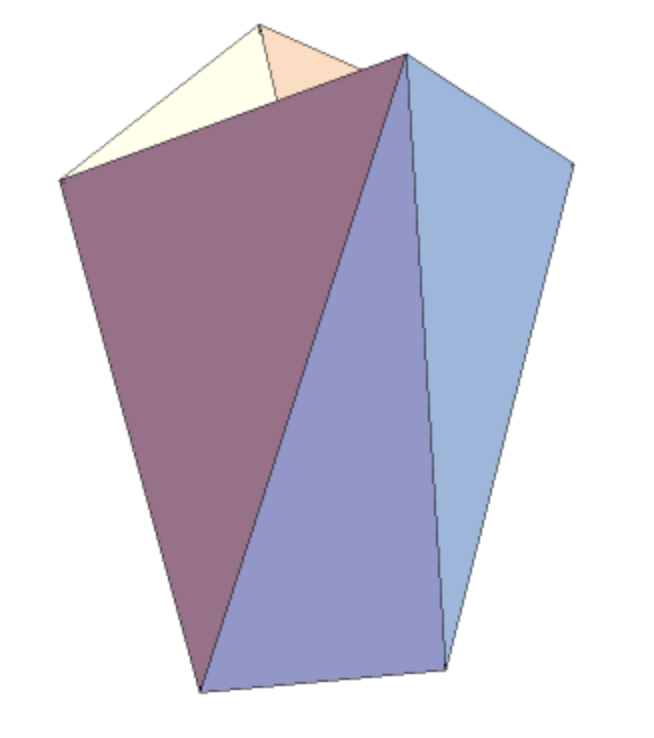}
  \caption{Weak convexity is lost}
  \label{fig:sfig5}
\end{subfigure}
\caption{A non weakly convex example}
\label{fig:fig5}
\end{figure}

Things look quite different when investigating the polyhedron obtained by "pushing" the top vertex of the previous one towards the interior. Even though 
the triangulations of both polyhedra admit the same number of simplices, the first one is (weakly) convex whereas the second one is definitely not.
In the latter case, we find that $M_T$ admits a negative real and zero as eigenvalues, which implies, using Lemma
\ref{L_1}, that the polyhedron is not infinitesimally rigid, hence infinitesimally flexible.
So, weak convexity is not an assumption 
that should be dropped from the main conjecture. 

Coming back to $T$-polyhedra, it is indeed the case that the triangulation issue we encountered before can be resolved, and that by shifting the Schönhardt hull to an appropriate height 
(for a given exterior hull, there are two possible choices
for the height of the Schönhardt hull that would make its $\Delta h$  be identical to the $\Delta h$  of the exterior hull).

\begin{figure}[H]
\centering
\includegraphics[width=6cm,height=5cm]{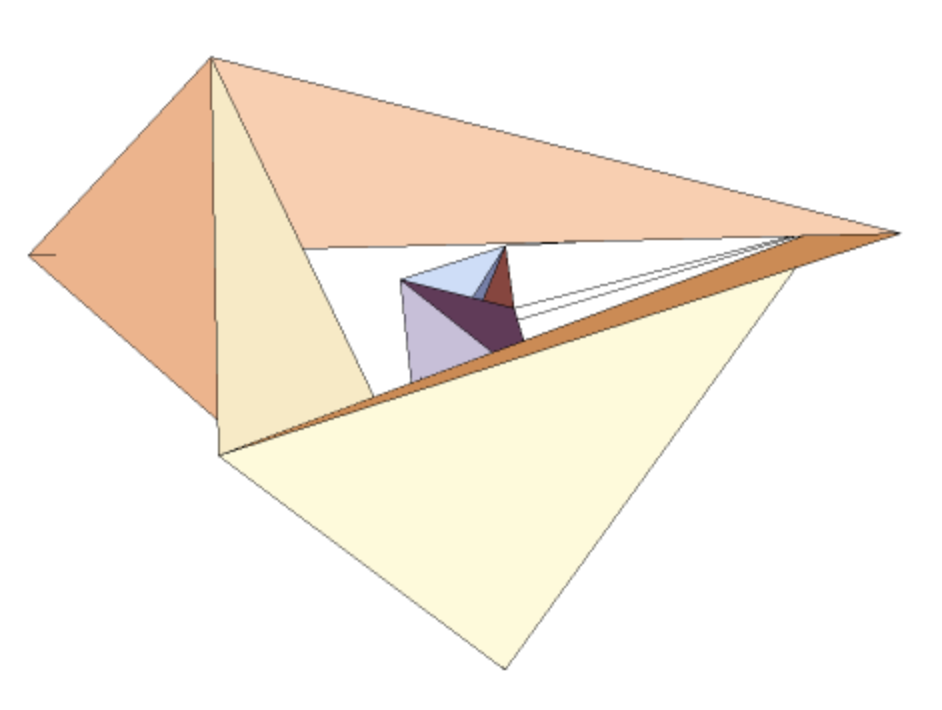}
\caption{Resolving non-decomposability}
\label{ST3}
\end{figure}
Unfortunately, this process makes the polyhedron lose weak convexity and admittedly, it does seem just like flexibility forces us to choose between decomposability or weak convexity, an exclusive or.
Thus, let us make one final (and quite reckless) \textbf{experimental conjecture}: 
\begin{center}
    $\textbf{EC1}$: There is no $T$-polyhedron that twists, 
    that is,  any weakly convex, 
    decomposable polyhedron having 
    the Schönhardt polyhedron 
    as its complement is 
    infinitesimally rigid. 
\end{center}

\subsection{Conclusions and outlook}
As 
the reader may have remarked, the conducted \textit{Mathematica} - calculations (although very interesting) did not provide us with any additional information. 
Taking into account that the polyhedra  studied in this 
paper were rather small (none of them
exceeded $18$ vertices) and  predictable, most of the computations
 could have been replaced by a straightforward application of a certain theorem and we were not  able to develop the powerful formalism involving $M_T$ to its full potential.
 This could of course be achieved by writing a
 better program that allows to test millions of
 more complicated examples. 
Maybe such 
 an expanded and more sophisticated 
 search would lead towards the almighty one that might disprove the main conjecture. \\ \\
 We hope that this note will be considered helpful 
 then.

\appendix

\section{Finding the eigenvalues of $M_T$ with \textit{Mathematica}}
In order to demonstrate how \textit{Mathematica} extracts 
the eigenvalues of  $M_T$, it 
is best to study a simple 
example.
The first step consists in defining the combinatorics of the polyhedron. 
For instance, an octahedron 
can be encoded as:
\begin{verbatim}
    Q =
Polyhedron[{A1 = {1,0,0}, A2 = {-1,0,0}, A3 = {0,1,0},
A4 = {0,-1,0}, A5 = {0,0,1}, A6 = {0,0,-1}},
{{1,3,5},{2,5,4},{4,1,5},{6,1,3},
{3,6,2},{2,6,4},{4,6,1},{3,2,5}}]
\end{verbatim}

Since this octahedron can be triangulated by means of four equivalent simplices only  one  dihedral angle needs
to be determined, the remaining ones being identical in this case. 
Thus, we choose the vertices A2, A4, A5, and A6, and compute the 
resulting Cayley-Menger determinant:
\begin{verbatim}
    CM1 =
Det[{{0,EuclideanDistance[A2,A5]^2, EuclideanDistance[A5,A4]^2, 
        EuclideanDistance[A5,A6]^2,1},
    {EuclideanDistance[A2,A5]^2, 0, EuclideanDistance[A2,A4]^2, 
        EuclideanDistance[A2,A6]^2, 1},
    {EuclideanDistance[A5,A4]^2, EuclideanDistance[A2,A4]^2, 
        0, EuclideanDistance[A4,A6]^2, 1},
    {EuclideanDistance[A5,A6]^2, EuclideanDistance[A2,A6]^2, 
        EuclideanDistance[A4,A6]^2, 0,1},
    {1,1,1,1,0}}]
\end{verbatim}
Now, in order to use Equation \eqref{eq1}, one 
starts by computing the minor of the matrix CM1 associated to the edge length of the simplex that corresponds to the interior edge length of the triangulation. 
Concretely, this is the edge between vertices A5 and A6. 
Then,
\begin{verbatim}
    De2 =
Det[{{0, EuclideanDistance[A5,A4]^2, EuclideanDistance[A5,A6]^2,1}, 
    {EuclideanDistance[A2,A5]^2, EuclideanDistance[A2,A4]^2, 
        EuclideanDistance[A2,A6]^2, 1}, 
    {EuclideanDistance[A5,A4]^2, 0, EuclideanDistance[A4,A6]^2, 1}, 
    {EuclideanDistance[A5,A6]^2, EuclideanDistance[A4,A6]^2, 0,1}}]
\end{verbatim}
in combination with Equation \eqref{eq1} and 
\begin{verbatim}
    ArcCos[De2/(Sqrt[2 * EuclideanDistance[A5,A6]^2 * CM + (De2)^2])]
\end{verbatim}
yields 
the expected dihedral angle of $\alpha_0 := \pi/2$. 
This is the initial dihedral angle.

Of course, the total angle around the edge $A5,A6$ is the sum of the individual dihedral angles.
Since they are all equal, we obtain $4 \cdot \pi/2 = 2\pi$, which  is not much of a 
surprise.
In particular, since we haven't  deformed the metric inside of the polyhedron yet, everything is nice and Euclidean and dihedral angles naturally  sum up to $2\pi$. A useful fact to remember. 

Recalling that the entries of $M_T$ are the derivatives of the total angle around each edge  with respect to the 
corresponding edge length, we are faced with a
complication since
 differentiating the expression obtained by Equation 
\eqref{eq1} slows 
the calculations down considerably.
Thus as a first step and in order to facilitate computations, we'll make the following approximation:
$$  \pdv{\omega_i}{l_j} \approx \frac{\omega_{if} - \omega_{i0}}{l_{jf} - l_{j0}}. $$ 
In other words, the interior edge lengths of the triangulation receive a tiny length variation, say $\epsilon$ (of course, $\epsilon >0 $), which (eventually) induces a change in the total angle around the modified edge and 
all the remaining  interior edges as well.
As remarked earlier, prior to the change of interior edge lengths the geometry is perfectly Euclidean. Hence $\omega_{i0} = 2\pi$, even without calculating any CM determinants.
In that vein, we can reformulate the approximated derivative as
$$\frac{\omega_{if} - \omega_{i0}}{l_{jf} - l_{j0}}  = \frac{\omega_{if} - 2 \pi}{ l_{j0} + \epsilon - l_{j0}} =  \frac{\omega_{if} - 2 \pi}{ \epsilon }, $$
which gains evermore on accuracy the smaller the $\epsilon$ is. 
The only quantity that is left to be determined is now $\omega_{if}$.
Coming back to our example, we can pick $\epsilon = 0.00000001$ and repeat the process from before while taking care that the edge length between vertices $A5$ and $A6$ has now gained on length 
(all the other edge lengths are kept constant). The Cayley-Menger determinant in this case is 
\begin{verbatim}
    CM2 =
Det[{{0, EuclideanDistance[A2,A5]^2, EuclideanDistance[A5,A4]^2, 
(EuclideanDistance[A5,A6] + 0.00000001)^2, 1},
{EuclideanDistance[A2,A5]^2, 0, EuclideanDistance[A2, A4]^2, 
EuclideanDistance[A2,A6]^2, 1}, 
{EuclideanDistance[A5,A4]^2, EuclideanDistance[A2, A4]^2, 
0, EuclideanDistance[A4,A6]^2, 1},
{(EuclideanDistance[A5,A6] + 0.00000001)^2, 
EuclideanDistance[A2,A6]^2, EuclideanDistance[A4, A6]^2, 0, 1},
{1,1,1,1,0}}]
\end{verbatim}
and with the appropriate minor
\begin{verbatim}
    De3 =
Det[{{0, EuclideanDistance[A5,A4]^2, 
        (EuclideanDistance[A5,A6] + 0.00000001)^2,1}, 
    {EuclideanDistance[A2,A5]^2, EuclideanDistance[A2,A4]^2, 
        EuclideanDistance[A2,A6]^2, 1},
    {EuclideanDistance[A5,A4]^2, 0, EuclideanDistance[A4,A6]^2, 1}, 
    {(EuclideanDistance[A5,A6] + 0.00000001)^2, EuclideanDistance[A4,A6]^2,0,1}}]
\end{verbatim}
one can use once again 
Equation \eqref{eq1} to obtain 
\begin{verbatim}
    ArcCos[De3/(Sqrt[2 * EuclideanDistance[A5, A6]^2 * CM2 + (De3)2])] = 1.5708.
\end{verbatim}
With this, the matrix $M_T$ becomes
$$M_T = \left( \frac{\omega_{if} - 2 \pi}{ \epsilon }  \right) = \left( \frac{\omega_{if} - 2 \pi}{ 0.00000001 }  \right) = \left( \frac{4 \cdot 1.5708  - 2 \pi}{ 0.00000001 }  \right) = \left( 1469.28 \right). $$
It being a $1\times1$ matrix in this example, the eigenvalues are easily read off and we can conclude, with the aid of Lemma \ref{L_1}, that the polyhedron is indeed infinitesimally rigid. 

Since this particular polyhedron is convex, it is true that we could have just applied Cauchy's theorem to conclude the same, and that without having to do any calculations.
However, the purpose of this example was to give a simple outline of the method, nothing more and nothing less. \\ \\
\textbf{Acknowledgments}\\
This paper is the result of an independent research project conducted under the Experimental Mathematics Lab during the winter semester of $2021$ at the University of Luxembourg  \url{https://math.uni.lu/eml/}. \\
I want to express my gratitude towards my supervisor Jean-Marc Schlenker for guiding me
through the realm of polyhedral geometry and Nina Morishige for many thought-provoking discussions.

\end{document}